\documentclass[preprint,12pt]{elsarticle}

\usepackage{amsmath}
\usepackage{amssymb}
\usepackage{amsthm}
\usepackage{graphicx}
\usepackage{bm}
\usepackage{mathrsfs}
\usepackage{tinywncy}
\usepackage{url}

\theoremstyle{plain} 
\newtheorem{theorem}{Theorem}[section] 
\newtheorem{lemma}[theorem]{Lemma}
\newtheorem{corollary}[theorem]{Corollary}
\newtheorem{proposition}[theorem]{Proposition}
\theoremstyle{definition} 
\newtheorem{definition}[theorem]{Definition}
\newtheorem{example}[theorem]{Example}
\theoremstyle{remark}
\newtheorem{remark}[theorem]{Remark}

\newcommand{\li}{\mathop{\mathrm{li}}\nolimits}
\newcommand{\Spec}{\mathop{\mathrm{Spec}}\nolimits}

\newcommand{\wt}{\mathop{\mathrm{wt}}\nolimits}

\newcommand{\Li}{\mathop{\mathrm{Li}}\nolimits}
\def\padic{{p\mathchar`-\mathrm{adic}}}
\begin{document}
	
	\title{Finite and \'etale polylogarithms}
	\renewcommand{\thefootnote}{\fnsymbol{footnote}}
	\author{Kenji Sakugawa$^{\text{a}}$, Shin-ichiro Seki$^{\text{b},}$\footnote{\textsc{Corresponding author} \\ {\it E-mail addresses}: k-sakugawa@cr.math.sci.osaka-u.ac.jp (S. Sakugawa), shinchan.prime@gmail.com (S. Seki).}}
	\address{{\tiny ${}^{\text{a}}$Department of Mathematics, Graduate School of Science Osaka University Toyonaka, Osaka 560-0043 Japan \\ ${}^{\text{b}}$Department of Mathematics, Graduate School of Science Osaka University Toyonaka, Osaka 560-0043 Japan}}
	
	\address{MSC: 11G55}
	
	\begin{abstract}
	We show an explicit formula relating \'etale polylogarithms introduced by Wojtkowiak and finite polylogarithms introduced by Elbaz-Vincent and Gangl. This formula is an \'etale analog of Besser's formula relating Coleman's $p$-adic polylogarithms and the finite polylogarithms.
	\end{abstract}
	
	\maketitle
	\section{Introduction}
	The polylogarithm appears in many fields of mathematics in different forms.
	The classical one is defined by the infinite sum
	\begin{equation}
		\label{classical}
		\mathrm{Li}_{m}(z):=\sum_{n=1}^\infty\frac{z^n}{n^m},\quad z\in \mathbb C
	\end{equation}
	for each positive integer $m$. The infinite sum (\ref{classical}) converges absolutely if
	the absolute value of $z$ is less than $1$.
	This holomorphic function has an analytic continuation to $\mathbb C$ as {\em a multi-valued} holomorphic function.
	We can define {\em a single-valued} version $\mathscr L_m(z)$ of $\Li_m(z)$, which is a natural generalization
	of the Bloch--Wigner function (\cite[p.413 (33)]{Z}).
	Before stating our main result ($=$ Theorem \ref{mthm2}), we recall various analogs of $\mathrm{Li}_{m}(z),\mathscr L_m(z)$ and recall relations between them. In this paper, $p$ denotes a prime number.
	\subsection{Coleman's $p$-adic polylogarithms}\label{Coleman}
	Let $\mathbb C_p$ be the $p$-adic completion of an algebraic closure of $\mathbb Q_p$.
	We denote by $|\ |_p$ the valuation
	on $\mathbb C_p$ normalized as $|p|_p=p^{-1}$.
	Then, for $z\in \mathbb C_p$ with $|z|_p<1$, the infinite sum (\ref{classical}) converges in $\mathbb C_p$
	and we denote this infinite sum by $\Li_m^\padic(z)\in \mathbb C_p$.
	In \cite{Col}, Coleman defined an analytic continuation of $\Li_m^\padic$ to $\mathbb C_p\setminus\{1\}$
	as so-called {\em a Coleman function}.
	We can define two Coleman functions $\mathscr L_m^\padic$ and $F_m^\padic$ as $p$-adic analogs of $\mathscr L_m$ by using $\Li_m^\padic$.
	For the definitions of them, see \ \cite[Introduction]{NSW2} and Subsection \ref{Finitepol} below, respectively.
	It is known that $\mathscr L_m^\padic$ and $F_m^\padic$ satisfy functional equations similar to $\mathscr L_m$
	(cf.\ \cite{Wo3}).
	\subsection{Wojtkowiak's $p$-adic \'etale polylogarithms}\label{petale}
	Wojtkowiak introduced a $p$-adic \'etale analog of $\mathscr L_m$
	which is called {\em a $p$-adic \'etale polylogarithm} (cf.\ \cite[Definition 11.0.1]{W1}).
	Let $K$ be a number field, $z\in K$, and $\gamma$ a $\mathbb Q_p$-path in $\mathbb P^1_{\overline K}\setminus\{0,1,\infty\}$
	from the tangential base point $\overrightarrow{01}$ to $z$. Then, the value at $z$ of $p$-adic \'etale polylogarithm $\mathrm{li}_m(z,\gamma)$ attached to $\gamma$
	is defined as a continuous function
	\begin{equation*}
		\mathrm{li}_m(z,\gamma)\colon G_K:=\mathrm{Gal}(\overline K/K)\rightarrow \mathbb Q_p(m),
	\end{equation*}
	where $\overline K$ is an algebraic closure of $K$ and $\mathbb Q_p(m)$ is the $m$-th Tate twist of the trivial Galois representation $\mathbb Q_p$.
	There is also a ``multiple version'' of the $p$-adic \'etale poloylogarithms called {\em $p$-adic \'etale iterated integrals}.
	See \cite{W0} for the precise definition.
	
	Nakamura, Wojtkowiak, and the first author showed an explicit formula relating Coleman's $p$-adic polylogarithms and Wojtkowiak's $p$-adic \'etale polylogarithms
	in \cite{NSW1}, \cite{NSW2}.
	Let $\mathbb{Z}[K\setminus\{0,1\}]$ be the free abelian group generated by the set $K\setminus\{0,1\}$ and take an element $\xi=\sum_{i=1}^n a_i\{z_i\}\in \mathbb{Z}[K\setminus\{0,1\}]$.
	Then, under some ``Bloch condition'' for $m$, the linear sum
	\[
	\sum_{i=1}^n a_i\mathrm{li}_m(z_i,\gamma_i)\colon G_K\rightarrow \mathbb Q_p(m)
	\]
	becomes a $1$-cocycle for some specific $\mathbb{Q}_p$-paths $\gamma_i$ (cf.\ \cite{NSW2}).
	Let $v\colon K\hookrightarrow \mathbb C_p$ be a finite place of $K$ over $p$ and suppose that $m$ is greater than $1$.
	Then the Bloch--Kato logarithm induces a linear map $\log^{\mathrm{BK}}_v \colon H^1(K,\mathbb Q_p(m))\rightarrow {D}_{\mathrm{dR},K_v}(\mathbb Q_p(m))\cong K_v\subset \mathbb C_p$ where $K_v$ is the topological closure of $v(K)$ in $\mathbb C_p$
	and $D_{\mathrm{dR},K_v}$ is a Fontaine functor (cf.\ \cite{Fo}).
	In \cite{NSW2}, they proved the equality called {\em the polylogarithmic Coleman--Ihara formula}:
	\begin{equation}
		-\log_v^{\mathrm{BK}}\left(\sum_{i=1}^n a_i\mathrm{li}_m(z_i,\gamma_i)\right)=\sum_{i=1}^na_i \mathscr L_m^\padic(v(z_i)).
		\label{CIformula}
	\end{equation}
	\subsection{Finite polylogarithms}\label{Finitepol}
	There exists a ``finite analog'' $\text{\rm \pounds}_{p, m}(t)$ of $\mathrm{Li}_m(z)$ introduced by Elbaz-Vincent--Gangl based on Kontsevich's observation (cf.\ \cite{EG}, \cite{Ko}).
	This finite analog is a polynomial defined by truncating the infinite sum
	(\ref{classical}) at degree $p$, namely,
	\begin{equation*}
		\text{\rm \pounds}_{p, m}(t):=\sum_{n=1}^{p-1}\frac{t^n}{n^m} \in \mathbb{Z}_{(p)}[t],
	\end{equation*}
	where $\mathbb{Z}_{(p)}$ is the localization of $\mathbb{Z}$ at a prime ideal $p\mathbb{Z}$.
	Besser established an explicit formula connecting the derivative of $F_m^\padic$ with $\text{\rm \pounds}_{p,m}(t)$
	in \cite{B}. Let $z$ be an element of $W(\overline{\mathbb F_p})$ such that $|z|_p=|1-z|_p$, where $W(\overline{\mathbb F_p})$
	is the ring of Witt vectors of $\overline{\mathbb F_p}$. We define $F_m^\padic$  by 
	\[
	F_m^{\padic}:=\sum_{j=0}^{m-1}a_j\log_p^j(z)\Li^\padic_{m-j}(z)
	\]
	with $a_0=-m$ and
	\[
	a_j = \frac{(-1)^j}{(j-1)!}+\frac{(-1)^{j+1}m}{j!}
	\]
	for $j > 0$ where $\log_p$ is a $p$-adic logarithm. Then, Besser showed the following formula
	\begin{equation}
		\label{Besserformula}
		p^{1-m}DF_m^\padic(z)\equiv \text{\rm \pounds}_{p, m-1}(\mathrm{Fr}_p (z)) \pmod{p}
	\end{equation}
	for $m$ greater than $1$ (\cite[Theorem 1.1]{B}). Here,  $\mathrm{Fr}_p \colon \overline{\mathbb{F}_p} \to \overline{\mathbb{F}_p}$ is the geometric Frobenius automorphism and $D$ is a differential operator defined by $D:=z(1-z)d/dz$.
	
	In \cite{EG}, Elbaz-Vincent--Gangl gave a procedure deriving functional equations of finite polylogarithms from functional equations of $p$-adic polylogarithms. The congruence (\ref{Besserformula}) plays a key role in their paper.
	
	\subsection{Main result}\label{mainresult}
	In this paper, we give an explicit formula of $\li_{p,m}(z,\gamma)$ relating it with a finite polylogarithm.
	When $p$ is greater than $m-1$ and $\gamma$ is an \'etale path, the
	image of $\li_{p,m}(z,\gamma)$ is contained in $\mathbb Z_p(m)$.
	We denote by $\text{\rm \pounds}^{\mathrm{\acute{e}t}}_{p,m} (z,\gamma)$
	the composite of $\li_{p,m}(z,\gamma)$ and the natural projection
	$\mathbb Z_p(m)\twoheadrightarrow \mathbb F_p(m) $.
	For each finite place $v$ of $K$ over $p$, we define a local field $K_{v,z}$ to be $K_v(\zeta_p, z^{1/p})$
	where $\zeta_p$ is a fixed primitive $p$-th root of unity.
	Let $\Lambda_{v,z}:=\log_p(\mathcal O_{K_{z,v}}^\times)/p\log_p(\mathcal O_{K_{z,v}}^\times)$.
	To calculate $\text{\rm \pounds}^{\mathrm{\acute{e}t}}_{p,m} (z,\gamma)$,
	we will introduce a group homomorphism
	\[
	\lambda_{v,z}^{(m)}\colon H^1_f(K_{v,z},\mathbb F_p(1))\otimes_{\mathbb F_p}\mathbb F_p(m-1)\rightarrow \Lambda_{v,z}\otimes_{\mathbb F_p}\mathbb F_p(m-1),
	\]
	where $H^1_f(K_{v,z},\mathbb F_p(1))$ is the finite part of $H^1(K_{v,z},\mathbb F_p(1))$
	(cf.\ Definition \ref{dfn4.8}).
	We show that the restriction of $\text{\rm \pounds}^{\mathrm{\acute{e}t}}_{p,m} (z,\gamma)$ to $G_{K_{v, z}}$ is an element of $H^1_f(K_{v,z},\mathbb F_p(1))\otimes_{\mathbb F_p}\mathbb F_p(m-1)$ (Proposition \ref{propB}). The main result of this paper is as follows:
	\begin{theorem}[{= Theorem \ref{thmC}}]  Let us take the same notation as above. We assume that $m$ is greater than $1$.
		Then, the following congruence holds in $\Lambda_{v,z}\otimes_{\mathbb{F}_p}\mathbb{F}_p(m-1):$
		\begin{equation}
			\label{meq2}\lambda_{v,z}^{(m)}\left( \text{\rm \pounds}^{\mathrm{\acute{e}t}}_{p,m} (z,\gamma)\right)\equiv
			\left[ \frac{(1-\zeta_{ p})^{{p}-m}}{z-1}\text{\rm \pounds}_{p,m}\left(z^{1/p}\right)\right]\otimes\zeta_{ p}^{\otimes{(m-1)}}\ \ \ \Bigl( \mathrm{mod}\ (\zeta_{p}-1)^{{p}-m+1}\Bigr)\nonumber
		\end{equation}
		for a sufficiently large prime number $p$.
		\label{mthm2}
	\end{theorem}
	The formula in Theorem \ref{mthm2} can
	be regarded as {\em an \'etale analog} of Besser's formula (\ref{Besserformula}).
	Furthermore, it also can be regarded as {\em a finite analog} of the polylogarithmic Coleman--Ihara formula
	(\ref{CIformula}).
	The proof of Theorem \ref{mthm2} is based on an explicit formula of finite \'etale polylogarithms proved by Nakamura--Wojtkowiak (Proposition \ref{propNW}) and a functional equation of finite star-multiple polylogarithms (see \ref{sec:app} for the definition). Though finite {\em star-multiple} polylogarithms do not appear in the statement of Theorem \ref{mthm2}, they unexpectedly appear in the proof (see the proof of Proposition \ref{propD}). 
	
	In Section \ref{sec:reformulation}, we reformulate our main result with an ad\'ele-like language according to \cite{KZ}, \cite{SS}.
	\subsection{Notation}\label{notation}
	Let $K$ be a field of characteristic $0$. We denote by $\mu (K)$ the set of roots of unity in $K$. We fix an algebraic closure $\overline{K}$ of $K$ and the symbol $G_K$ denotes the absolute Galois group $\mathrm{Gal}(\overline K/K)$ of $K$.
	Let $L$ be a local field or an algebraic extension of $\mathbb Q$. Then, we denote by $\mathcal O_L$
	the ring of integers of $L$.
	For each locally noetherian affine scheme $\mathrm{Spec}(R)$ and for each topological abelian group $A$ equipped with a continuous action of the \'etale fundamental group $\pi:=\pi_1^{\text{\'et}}(\mathrm{Spec}(R))$ of $\mathrm{Spec}(R)$, we denote by $H^i_{\mathrm{cont}}(R,A)$ the $i$-th continuous group cohomology
	$H^i(\pi,A)$.
	Note that $H^i_{\mathrm{cont}}(R,A)$ coincides with the $i$-th continuous \'etale
	cohomology in the sense of Jannsen (cf.\ \cite{Jann}) if $A$ is a profinite group and $R$ is a product of fields.
	If $A$ is a finite group, then we omit the notation ``cont'' from our notation
	because $H^i_{\mathrm{cont}}(R,A)$ coincides with the usual \'etale cohomology
	group of the \'etale sheaf attached to $A$ on $\Spec(R)_{\text{\'et}}$.
	We denote by
	\[
	\kappa_{n, K} \colon K^\times/(K^\times)^{n}\xrightarrow{\sim} H^1(K,{{\mathbb Z}}/{n{\mathbb Z}}(1))
	\]
	the Kummer map induced by the Kummer sequence
	\[
	1\rightarrow {{\mathbb Z}}/{n{\mathbb Z}}(1)\rightarrow \overline K^\times\xrightarrow{n} \overline K^\times \rightarrow 1
	\]
	where ${{\mathbb Z}}/{n{\mathbb Z}}(1)$ is the group of $n$-th roots of unity in $\overline K$.
	\section{Review of \'etale polylogarithms}
	\label{sec:Finite etale plylogarithms and finite polylogarithms}
	In this section, we review \'etale polylogarithms introduced by Wojtkowiak.
	Let $p$ be a prime number. Suppose that $K$ and $\overline{K}$ are subfields of $\mathbb C$.
	Let $X$ be ${\mathbb P}^1_K\setminus\{0,1,\infty\}$ and $\pi_1^{p}(X)$ the maximal pro-$p$ quotient of the \'etale fundamental
	group $\pi_1^{\mathrm{\acute{e}t}}(X\otimes_K \overline K,\overrightarrow{01})$.
	Let us take a standard set of generators $\{x,y\}$ of $\pi_1^{\mathrm{\acute{e}t}}(X\otimes_K \overline K,\overrightarrow{01})$
	as in \cite[Section 2.3]{Ih90}.
	Then, $x$
	determines a coherent system of roots of unity $(\zeta_n)_{n>0},\zeta_n\in {\mathbb C}$.
	We regard this coherent system as a basis of $\widehat {\mathbb Z}(1):=\varprojlim_n \mathbb{Z}/n\mathbb{Z}(1)$.
	We define a multiplicative embedding
	\[
	\iota \colon \pi_1^{p}(X)\hookrightarrow \mathbb Q_p\langle \! \langle A,B \rangle \! \rangle
	\]
	of $\pi_1^{p}(X)$ to the non-commutative formal power series over $\mathbb Q_p$ in variables $A,B$
	by
	\[
	\iota(x):=\exp(A),\quad \iota(y):=\exp(B).
	\]
	For $z\in K$ and for an \'etale path $\gamma$ in $X\otimes_K \overline K$ from $\overrightarrow{01}$ to $z$, we consider the continuous $1$-cocycle
	\[
	\mathfrak f_\gamma \colon G_K\rightarrow \pi_1^{p}(X);\quad \sigma\mapsto \gamma^{-1}\circ {^\sigma \gamma}
	\]
	(cf.\ \cite[(2.3.1)]{Ih90}).
	Note that the constant term of $\iota(\mathfrak f_\gamma(\sigma))$
	is equal to $1$ for any $\sigma \in G_K$.
	Hence, the infinite sum
	\[
	\log(\iota(\mathfrak f_\gamma(\sigma))):=\sum_{n=1}^\infty \frac{(-1)^{n-1}(\iota(\mathfrak f_\gamma(\sigma))-1)^n}{n}
	\]
	converges in $\mathbb Q_p\langle \! \langle A,B \rangle \! \rangle$.
	We define a special element $e_m$ for each positive integer $m$
	by
	\[
	e_1:=B,\quad e_m:=[A,e_{m-1}]\in \mathbb Q_p\langle\!\langle A,B\rangle\!\rangle.
	\]
	Here, for each $\xi, \eta \in \mathbb{Q}_p\langle \! \langle A, B \rangle \! \rangle$, the bracket product $[\xi, \eta]$ is defined to be $\xi \eta - \eta \xi$.
	\begin{definition}[\cite{W0}]
		Let $m$ be a positive integer, $z$ an element of $K\setminus\{0,1\}$,
		and $\gamma$ an \'etale path from $\overrightarrow{01}$ to $z$.
		Then, the value at $z$ of {\em the $p$-adic \'etale polylogarithm}
		\[
		\mathrm{li}_{p,m}^{\mathrm{\acute{e}t}}(z,\gamma)\colon G_K\rightarrow \mathbb Q_{p}(m)
		\]
		attached to $\gamma$ is defined by
		\[
		\mathrm{li}_{p, m}^{\mathrm{\acute{e}t}}(z,\gamma)(\sigma):=\left(\text{the coefficient of }\log\left(\iota(\mathfrak f_\gamma(\sigma))\right)\text{ at }e_m\right)\otimes(\zeta_{p^n})_{n>0}^{\otimes m}.
		\]
		\label{dfn4.1}
	\end{definition}
	In general, these continuous functions are not $1$-cocycles.
	The following lemma is easily checked by construction:
	\begin{lemma}The image of $\mathrm{li}_{p,m}^{\mathrm{\acute{e}t}}(z,\gamma)$
			is contained in ${\mathbb Z_p}[\frac{1}{m!}](m)\subset \mathbb Q_{p}(m)$.
		\label{lem4.2}
	\end{lemma}
	\begin{remark}$ $
		\begin{itemize}
			\item[(1)]For each positive integer $m$, the continuous map $\mathrm{li}_{p,m}(z,\gamma)$
			coincides with Wojtkowiak's $p$-adic polylogarithm $l_m(z)_\gamma$ {\em multiplied by $(-1)^{m-1}$} (cf.\ \cite[Definition 11.0.1]{W1}).
			\item[(2)]Let $\pi_1^{\mathrm{nil}}(X)$ be the pro-nilpotent completion
			of the \'etale fundamental group of $X\otimes_K\overline K$. Then, by replacing $\pi_1^p(X)$ with $\pi_1^{\mathrm{nil}}(X)$, we can define an ad\'elic version \[
			\mathrm{li}_{\mathbb A_{\mathbb Q,f},m}^{\mathrm{\acute{e}t}}(z,\gamma)\colon G_K\rightarrow \mathbb A_{\mathbb Q,f}(m)
			\] of $\mathrm{li}_{p,m}^{\mathrm{\acute{e}t}}(z,\gamma)$.
			Here, $\mathbb A_{\mathbb Q,f}(m)$ is the $m$-th Tate twist of the ring of finite ad\'eles
			$\mathbb A_{\mathbb Q,f}$ and we call this continuous function the value at $z$ of {\it the \'etale polylogarithm} attached to $\gamma$.
		\end{itemize}
		\label{normalization}
	\end{remark}
	\begin{definition}
		For each positive integer $m$ less than $p$, we define the continuous map
		\[
		\text{\rm \pounds}_{p,m}^{\mathrm{\acute{e}t}}(z,\gamma) \colon G_K\rightarrow \mathbb F_p(m)
		\]
		to be the composite of $\mathrm{li}_{p, m}^{\mathrm{\acute{e}t}}(z,\gamma)$
		and the natural projection $\mathbb Z_p(m)\twoheadrightarrow \mathbb F_p(m)$.
		We call $\text{\rm \pounds}_{p,m}^{\mathrm{\acute{e}t}}(z,\gamma)$
		the value at $z$ of {\em the mod $p$ \'etale polylogarithm} attached to $\gamma$.
		\label{dfn4.2}
	\end{definition}
	For later use, we define a homomorphism $\lambda_{\mathbb F_p(m-1),K}$. 
	\begin{definition}Let $K$ be a finite extension of $\mathbb Q_p$.
		\begin{itemize}
			\item[(1)]Let $\log_p\colon \mathcal O_K^\times\to K$ be a $p$-adic logarithm.
			Then, we define $\Lambda_K$ to be $\log_p(\mathcal O_K^\times)/p\log_p(\mathcal O_K^\times)$.
			By definition, $\log_p$ induces a natural homomorphism \[
			\log_{\mathbb F_p(m-1),K}\colon \mathcal O_K^\times/(\mathcal O_K^\times)^p\otimes_{\mathbb F_p}\mathbb F_p(m-1)\to \Lambda_K(m-1):=\Lambda_K\otimes_{\mathbb F_p}\mathbb F_p(m-1).
			\]
			\item[(2)] We define a group homomorphism
			\begin{eqnarray*}
				\lambda_{\mathbb F_p(m-1),K}\colon H^1_f(K,\mathbb F_p(1))\otimes_{\mathbb F_p}\mathbb F_p(m-1)\rightarrow \Lambda_{K}(m-1)
			\end{eqnarray*}
			by composing $\kappa^{-1}_{p, K}$ and $\log_{\mathbb F_p(m-1),K}$. Here, $H^1_f(K,\mathbb F_p(1))$
				is the subgroup of $H^1(K,\mathbb F_p(1))$ defined to be the image of $\mathcal O_K^\times/(\mathcal O_K^\times)^p$ under the Kummer map.
		\end{itemize}
		\label{dfn4.8}
	\end{definition}
	
	\section{Main result}
	\label{subsec:Explicit formula of finite etale polylogatihms}
	In this and the next section, we fix a number field $K$ contained in $\mathbb C$
	and fix $z\in K\setminus\{0, 1\}$.
	For a finite place $v$ of $K$ dividing $p$, we put
	\[
	K_{v,z}:=K_v( \mu_p,z^{1/p} ),
	\]
	where $K_v$ is the $v$-adic completion of $K$ and $\mu_p$ is the set of all $p$-th roots of unity in $\overline{K}$. 
	\begin{proposition}
		[{\cite[Section 3, Corollary]{NW}}]
		Let $m$ be a positive integer less than $p-1$. Let $z$ be an element of $K\setminus\{0,1\}$
		and $\gamma$ an \'etale path in ${\mathbb P}^1_{\overline K}\setminus\{0,1,\infty\}$ from $\overrightarrow{01}$ to $z$. Let $z^{1/p}$ be the $p$-th root of $z$ determined by $\gamma$ and put
		\[
		w_{p,m}(z,\gamma):=\prod_{i=0}^{p-1}\left(1-z^{1/p}\zeta_p^i\right)^{i^{m-1}},
		\]
		where we understand  $0^0=1$ when $m=1$ and $i=0$. Then, the equality
		\[
		\left.\text{\rm \pounds}^{\mathrm{\acute{e}t}}_{p,m}(z,\gamma)\right|_{G_{K_{v, z}}}=\frac{1}{(m-1)!}\kappa_{p,K_{v, z}}(w_{p,m}(z,\gamma))\otimes \zeta_p^{\otimes(m-1)}
		\]
		holds in $H^1(K_{v, z},\mathbb F_p(m))=H^1(K_{v, z},\mathbb F_p(1))\otimes_{\mathbb F_p}\mathbb F_p(m-1)$.
		\label{propNW}
	\end{proposition}
	\begin{proposition}
		Let us take the same notation as in Proposition $\ref{propNW}$
		and let $v$ be a finite place of $K$ dividing $p$.
		Then, the restriction of the mod $p$ \'etale polylogarithm $\text{\rm \pounds}^{\mathrm{\acute{e}t}}_{p,m} (z,\gamma)$ to the absolute Galois group $G_{K_{v,z}}$ is a continuous group homomorphism.
		Furthermore, it is contained in $H^1_f(K_{v,z},\mathbb F_p(1))\otimes_{\mathbb F_p}\mathbb F_p(m-1)$ if $p$ does not divide $1-z$.
		\label{propB}
	\end{proposition}
	\begin{proof}
		By Proposition \ref{propNW}, it is easy to show that the restriction is a group homorphism.
		We suppose that $p$ does not divide $1-z$. Then, $w_{p,m}(z,\gamma)$ is a $p$-unit and the restriction is contained in the finite part.
	\end{proof}
	\begin{remark}
		By Proposition \ref{propNW},
		if $z$ is contained in $\mu(K)\setminus\{1\}$,
		then $\text{\rm \pounds}^{\mathrm{\acute{e}t}}_{p,m} (z,\gamma)$
		is the mod $p$ of {\em a cyclotomic Soul\'e element} multiplied by $1/(m-1)!$
		(cf.\ \cite[Corollary 14.3.3]{W1}).
		\label{rem1}
	\end{remark}
	We use the abridged notation $\lambda_{v, z}^{(m)}=\lambda_{\mathbb{F}_p(m-1), K_{v, z}}$. By abuse of notation, we denote the restriction to $G_{K_{v, z}}$ of $\text{\rm \pounds}^{\mathrm{\acute{e}t}}_{p,m} (z,\gamma)$ by the same notation $\text{\rm \pounds}^{\mathrm{\acute{e}t}}_{p,m} (z,\gamma)$. Then, the main result of this paper is as follows:
	\begin{theorem}Let $K$ be a number field, $z$ an element of $K\setminus\{0, 1\}$, and $\gamma$ an \'etale path
		in ${\mathbb P}^1_{\overline K}\setminus\{0,1,\infty\}$ from $\overrightarrow{01}$ to $z$.
		Let $v$ be a finite place of $K$ over $p$ unramified in $K/\mathbb Q$ and $m$ a positive integer such that $1 < m < p-1$.
		We denote by $z^{1/p}$ the $p$-th root of $z$ determined by $\gamma$ and assume that $p$ does not divide $z(1-z)$. 
		Then, we have
		\begin{equation}
			\begin{split}
			&\lambda_{v, z}^{(m)}\left( \text{\rm \pounds}^{\mathrm{\acute{e}t}}_{p,m} (z,\gamma)\right)=\\
			&\left[ \frac{(1-\zeta_{p})^{ {p}-m}}{z-1}\text{\rm \pounds}_{p,m}\left(z^{1/ p}\right)+\alpha_p(\zeta_{p}-1)^{{ p}-m+1}\bmod{\bigl(p\log_p(\mathcal{O}_{K_{v, z}}^{\times})\bigr)}\right]\otimes\zeta_{p}^{\otimes{(m-1)}}
			\end{split}
		\label{eqB}\end{equation}
		in $\Lambda_{K_{v, z}}(m-1)$ for a certain element $\alpha_p\in \mathcal O_{K_{v,z}}$.
		\label{thmC}
	\end{theorem}
	We prove this theorem in the next section.
	By Theorem \ref{thmC},
	we have a new proof of the following well-known fact.
	\begin{corollary}[cf.\ {\cite[Corollary 3.2]{So}}]
		Assume that $p$ is an odd regular prime. Then, for any odd integer $m$ such that $1<m<p-1$, the restriction map
		\[
		\mathrm{res}_m \colon H_{\mathrm{cont}}^1({\mathbb Z}[1/p],{\mathbb Z}_p(m))\rightarrow H_{\mathrm{cont}}^1({\mathbb Q}_p,{\mathbb Z}_p(m))
		\]
		is an isomorphism.
		\label{corA'}
	\end{corollary}
	This corollary will be proved in Section \ref{subsec:Appriations of Theorem C}.
	\section{Proof of Theorem \ref{thmC}}
	\label{sec:Proof of Theorem C}
	In this section, 
	we show the congruence (\ref{eqB})
	for fixed $m,z$, and $v$ satisfying the conditions of Theorem \ref{thmC}. Then, $K_v$ is unramified over ${\mathbb Q}_p$ by assumption.
	Theorem \ref{thmC} (\ref{eqB}) is a direct consequence of Proposition \ref{propNW} and the following proposition:
	\begin{proposition}
		We have the following congruence in $\mathcal O_{K_{v, z}}:$
		\[
		\log_p(w_{p,m} (z,\gamma))\equiv (m-1)!\frac{(1-\zeta_p)^{p-m}}{z-1}
		{\text{\rm \pounds}_{p,m}(z^{1/p})} \ \ \Bigl( \mathrm{mod}\ (\zeta_p-1)^{p-m+1}\Bigr).
		\]
	\label{propD}
	\end{proposition}
	We prepare some auxiliary lemmas to prove Proposition \ref{propD}.
	\begin{lemma}Let $n$ and $j$ be positive integers less than $p$. Then, the following congruence holds in ${\mathbb Z}_{(p)}[\zeta_p]:$
		\begin{equation*}
			\sum_{i=0}^{p-1}i^n\zeta_p^{ij}\equiv \frac{(-1)^{n-1}n!p}{j^n(\zeta_p-1)^n}\ \hspace{10mm} \left(\mathrm{mod}\ \frac{p}{(\zeta_p-1)^{n-1}}\right).
		\end{equation*}
		\label{lem1.5}
	\end{lemma}
	\begin{proof}We remark that $\zeta_p^j-1$ is congruent to $j(\zeta_p-1)$ mod $(\zeta_p-1)^2$. Hence, for some $a,b,c\in {\mathbb Z}_{(p)}[\zeta_p]$, we obtain the following equalities:
		\begin{equation*}
			\begin{split}
				\frac{1}{(\zeta_p^j-1)^n}&=\frac{1}{(j(\zeta_p-1)+a(\zeta_p-1)^2)^n}=\frac{1}{j^n(\zeta_p-1)^n(1+b(\zeta_p-1))}\\
				&=\frac{1}{j^n(\zeta_p-1)^n}(1+c(\zeta_p-1))=\frac{1}{j^n(\zeta_p-1)^n}+\frac{c}{j^n(\zeta_p-1)^{n-1}}.
			\end{split}
		\end{equation*}
		In particular, $p/(\zeta_p^j-1)^n$ is congruent to $p/(j^n(\zeta_p-1)^n)$ mod $p/(\zeta_p-1)^{n-1}$.
		Therefore, it is sufficient to show the following congruence:
		\begin{eqnarray}
			\sum_{i=0}^{p-1}i^n\zeta_p^{ij}\equiv \frac{(-1)^{n-1}n!p}{(\zeta_p^j-1)^n}\ \hspace{10mm} \left(\mathrm{mod}\ \frac{p}{(\zeta_p-1)^{n-1}}\right).
			\label{eq1}
		\end{eqnarray}
		We show the congruence (\ref{eq1}) by induction on $n$. 
		
		First, we prove the case $n=1$. We have the equalities
		\begin{equation*}
			\begin{split}
				(\zeta_p^j-1)\sum_{i=0}^{p-1}i\zeta_p^{ij}&=\sum_{i=0}^{p-1}i\zeta_p^{(i+1)j}-\sum_{i=0}^{p-1}i\zeta_p^{ij}\\
				&=\sum_{i=0}^{p-1}(i+1)\zeta_p^{(i+1)j}-\sum_{i=0}^{p-1}\zeta_p^{(i+1)j}-\sum_{i=0}^{p-1}i\zeta_p^{ij}\\
				&=\sum_{i=0}^{p-1}(i+1)\zeta_p^{(i+1)j}-\sum_{i=0}^{p-1}i\zeta_p^{ij}=p
			\end{split}
		\end{equation*}
		since $\sum_{i=0}^{p-1}\zeta_p^{(i+1)j}=0$. Thus, the assertion of the lemma holds for the case $n=1$.
		
		Next, we assume that $n$ is greater than $1$ and that the assertion holds for any positive integer less than $n$. Then, the following congruence holds:
		\begin{equation*}
			\begin{split}
				(\zeta_p^j-1)\sum_{i=0}^{p-1}i^n\zeta_p^{ij}&=\sum_{i=0}^{p-1}i^n\zeta_p^{(i+1)j}-\sum_{i=0}^{p-1}i^n\zeta_p^{ij}\\
				&=\sum_{i=0}^{p-1}(i+1)^n\zeta_p^{(i+1)j}-\sum_{i=0}^{p-1}\sum_{s=0}^{n-1}\binom{n}{s}i^s\zeta_p^{(i+1)j}-\sum_{i=0}^{p-1}i^n\zeta_p^{ij}\\
				&=\sum_{i=0}^{p-1}(i+1)^n\zeta_p^{(i+1)j}-\sum_{i=0}^{p-1}i^n\zeta_p^{ij}-\sum_{s=0}^{n-1}\binom{n}{s}\zeta_p^j\sum_{i=0}^{p-1}i^s\zeta_p^{ij}\\
				&\equiv p^n-\zeta_p^j\sum_{s=0}^{n-1}\binom{n}{s}\frac{(-1)^{s-1}s!p}{(\zeta_p^j-1)^{s}}\ \ \  \left(\mathrm{mod}\ \frac{p}{(\zeta_p-1)^{n-2}}\right).
			\end{split}
		\end{equation*}
		Note that $p/(\zeta_p^j-1)^{s}$ is divided by $p/(\zeta_p-1)^{n-2}$ for any non-negative integer $s$ less than $n-1$.
		Therefore, we obtain the congruences 
		\begin{equation*}
			(\zeta_p^j-1)\sum_{i=0}^{p-1}i^n\zeta_p^{ij}\equiv -\zeta_p^jn\frac{(-1)^{n-2}(n-1)!p}{(\zeta_p^j-1)^{n-1}} 
			\equiv \frac{(-1)^{n-1}n!p}{(\zeta_p^j-1)^{n-1}}\ \ \ \left(\mathrm{mod}\ \frac{p}{(\zeta_p-1)^{n-2}}\right).
		\end{equation*}
		This completes the proof of the lemma.
	\end{proof}
	\begin{lemma} The element $\log_p(1-z^{1/p})$ is contained in $\mathcal{O}_{K_{v, z}}$.
		\label{lem1.6} 
	\end{lemma}
	\begin{proof}
		We write $z=\zeta u$ with some $\zeta\in \mu(K_v)\setminus\{1\}$ and some principal unit $u \in 1+p\mathcal{O}_{K_v}$.
		Put  $\zeta':=\zeta^{1/p}/(1-\zeta^{1/p})$ and $\varpi:=u^{1/p}-1$.
		One can check that the $p$-adic valuation of $\varpi$ is greater than or equal to $1/p$.
		Then, $1-z^{1/p}=(1-\zeta^{1/p})(1-\zeta' \varpi)$ and we have 
		\[
		\log_p(1-z^{1/p})=\log_p(1-\zeta^{1/p})-\sum_{n=1}^\infty \frac{(\zeta'\varpi)^n}{n}.
		\]
		Let $v_p$ be the $p$-adic valuation of $K_{v, z}$ normalized as $v_p(p)=1$. Since $v_p((\zeta'\varpi)^n/n)= v_p(\varpi^n/n)\geq 0$ and the image of $\mathcal O_{K_v}^\times$ under $\log_p$
		is contained in $p\mathcal O_{K_v}$,
		we obtain the conclusion of the lemma.
	\end{proof}
	\begin{lemma}
		We have the following congruence in  ${\mathbb Z}_{(p)}[\zeta_p]:$
		\[
		(\zeta_p-1)^{p-1}\equiv -p\ (\mathrm{mod}\ p(\zeta_p-1)).
		\]
		\label{lem1.7}
	\end{lemma}
	\begin{proof}
		It follows from the fact that $\zeta_p-1$
		is a root of the polynomial $\left((X+1)^p-1\right)/X$.
	\end{proof}
	\begin{proof}[Proof of Proposition $\ref{propD}$]
		We put $\xi:=z^{-1/p}-1$.
		Then, $1-z^{1/p}\zeta_p^i$ is equal to $(1-z^{1/p})(1-\frac{\zeta_p^i-1}{\xi})$.
		Note that $\xi$ is a unit of $\mathcal O_{K_{v, z}}$ because $z\in \mathcal O_{K_v}^\times \cap(1+\mathcal O_{K_v}^\times)$.
		By the definition of $w_{p,m} (z,\gamma)$, we have the following equalities:
		{\small \begin{equation*}
				\log_p(w_{p,m} (z,\gamma))=\sum_{i=0}^{p-1}i^{m-1}\log_p(1-z^{1/p}\zeta_p^i)
				=\sum_{i=0}^{p-1}i^{m-1}\log_p(1-z^{1/p})+\sum_{i=0}^{p-1}i^{m-1}\log_p\left(1-\frac{\zeta_p^i-1}{\xi}\right).
			\end{equation*}
		}By Lemma \ref{lem1.6}, $\sum_{i=0}^{p-1}i^{m-1}\log_p(1-z^{1/p})$ is congruent to $0$ mod
		$p$.
		Hence, we have
		\[
		\log_p(w_{p,m} (z,\gamma))\equiv \sum_{i=0}^{p-1}i^{m-1}\log_p\left(1-\frac{\zeta_p^i-1}{\xi}\right)\ \pmod{p}.
		\]
		Meanwhile, the following equalities hold by the definition of the $p$-adic logarithm:
		\begin{equation}\label{eq7}
			\log_p\left(1-\frac{\zeta_p^i-1}{\xi}\right)=-\sum_{l=1}^\infty\frac{1}{l\xi^l}(\zeta_p^i-1)^l=-\sum_{p\nmid l}\frac{1}{l\xi^l}(\zeta_p^i-1)^l-\sum_{l=1}^\infty\frac{1}{pl\xi^{pl}}(\zeta_p^i-1)^{pl}.
		\end{equation}
		Note that, if $l$ is greater than $1$ (resp.\ greater than $p-1$), then the $p$-adic valuation of $\frac{(\zeta_p^i-1)^{pl}}{pl}$ (resp.\ $(\zeta_p^i-1)^l$) is greater than $1$. Thus the right hand side of the equality (\ref{eq7}) is congruent to $-\sum_{ l=1}^{p-1}\frac{1}{l\xi^l}(\zeta_p^i-1)^l-\frac{1}{p\xi^{p}}(\zeta_p^i-1)^p$ mod $p$ because $\xi\in \mathcal O_{K_{v, z}}^\times$. Furthermore, by Lemma \ref{lem1.5}, we obtain the following congruence:
		\begin{equation*}
			\begin{split}
				\frac{1}{p}\sum_{i=0}^{p-1} i^{m-1}(\zeta_p^i-1)^p&=\sum_{i=0}^{p-1}i^{m-1}\sum_{j=1}^{p-1}\frac{1}{p}\binom{p}{j}(-1)^{p-j}\zeta_p^{ij}\\
				&\equiv \sum_{j=1}^{p-1}\frac{1}{p}\binom{p}{j}(-1)^{p-j}\frac{(-1)^{m-2}(m-1)!p}{j^{m-1} (\zeta_p-1)^{m-1}}\ \ \ \left(\mathrm{mod}\ \frac{p}{(\zeta_p-1)^{m-2}}\right).
			\end{split}
		\end{equation*}
		Since $\frac{1}{p}\binom{p}{j}$ is congruent to $\frac{(-1)^{j-1}}{j}$ mod $p$,
		we have
		\begin{equation*}
			\frac{1}{p}\sum_{i=0}^{p-1} i^{m-1}(\zeta_p^i-1)^p \equiv \frac{(-1)^{m-2}(m-1)!p}{ (\zeta_p-1)^{m-1}}\sum_{j=1}^{p-1}\frac{1}{j^{m}}\equiv 0\ \hspace{9mm}\left(\mathrm{mod}\ \frac{p}{(\zeta_p-1)^{m-2}}\right)
		\end{equation*}
		as $p$ is greater than $m+1$ by assumption.
		Hence, we obtain the following congruence:
		\begin{equation}
			\sum_{i=0}^{p-1}i^{m-1}\log_p\left(1-\frac{\zeta_p^i-1}{\xi}\right)\equiv-\sum_{i=0}^{p-1}i^{m-1}\sum_{l=1}^{p-1}\frac{1}{l\xi^l}(\zeta_p^i-1)^l\label{eq3} \ \ \left(\mathrm{mod}\ \frac{p}{(\zeta_p-1)^{m-2}}\right).
		\end{equation}
		Now, let us calculate the right hand side of the congruence (\ref{eq3}). Since $\sum_{i=1}^{p-1}i^{m-1}$ is divided by $p$, we have the following congruence:
		\begin{equation*}
			\begin{split}
				\sum_{i=0}^{p-1}i^{m-1}(\zeta_p^i-1)^l&=\sum_{i=0}^{p-1}i^{m-1}\sum_{j=0}^l\binom{l}{j}(-1)^{l-j}\zeta_p^{ij}\\
				&=(-1)^l\sum_{j=0}^l\binom{l}{j}(-1)^j\sum_{i=0}^{p-1}i^{m-1}\zeta_p^{ij}\\
				&=(-1)^l\sum_{j=1}^l\binom{l}{j}(-1)^j\sum_{i=0}^{p-1}i^{m-1}\zeta_p^{ij}+(-1)^l\sum_{i=0}^{p-1}i^{m-1}\\
				&\equiv(-1)^l\sum_{j=1}^l\binom{l}{j}(-1)^j\sum_{i=0}^{p-1}i^{m-1}\zeta_p^{ij}\ \ (\mathrm{mod}\ p).
			\end{split}
		\end{equation*}
		Therefore, according to Lemma \ref{lem1.5}, we have the following congruence for any positive integer $l$ less than $p$:
		\begin{equation}
			\begin{split}
				\sum_{i=0}^{p-1}i^{m-1}(\zeta_p^i-1)^l&\equiv(-1)^l\sum_{j=1}^l\binom{l}{j}\frac{(-1)^{j+m-2}(m-1)!p}{j^{m-1}(\zeta_p-1)^{m-1}}\ \ \ \left(\mathrm{mod}\ \frac{p}{(\zeta_p-1)^{m-2}}\right)\\
				&=(-1)^l\frac{(-1)^{m-1}(m-1)!p}{(\zeta_p-1)^{m-1}}\sum_{j=1}^l\binom{l}{j}\frac{(-1)^{j-1}}{j^{m-1}}\\
				&=(-1)^l\frac{(-1)^{m-1}(m-1)!p}{(\zeta_p-1)^{m-1}}\sum_{l\geq n_1\geq \dots\geq n_{m-1}\geq 1}\frac{1}{n_1\cdots n_{m-1}}.\label{eq4}
			\end{split}
		\end{equation}
		Here, the last equality is obtained by Dilcher's identity (cf.\ \cite[Corollary 3]{D} or Theorem \ref{generalized Dilcher}).
		Finally, by substituting the congruence (\ref{eq4}) into the congruence (\ref{eq3}), we obtain the following congruence:
		\begin{equation}
			\begin{split}
				\log_p(w_{p,m} (z,\gamma))&\equiv-\frac{(-1)^{m-1}(m-1)!p}{(\zeta_p-1)^{m-1}}\sum_{l=1}^{p-1}\frac{(-1)^l}{l\xi^l}\sum_{l\geq n_1\geq \dots\geq n_{m-1}\geq 1}\frac{1}{n_1\cdots n_{m-1}} \\
				&\equiv\frac{(-1)^{m}(m-1)!p}{(\zeta_p-1)^{m-1}}\text{\rm \pounds}^{\star}_{p, \{1\}^{m}}\left(\frac{-1}{\xi}\right)\\
				&\equiv\frac{(-1)^{m}(m-1)!p}{(\zeta_p-1)^{m-1}}\text{\rm \pounds}^{\star}_{p, \{1\}^{m}}\left(\frac{1}{1-z^{-1/p}}\right) \hspace{10mm} \left( \bmod{\frac{p}{(\zeta_p-1)^{m-2}}} \right).
			\end{split}
		\label{log-FMP}\end{equation}
        Here, $\text{\rm \pounds}^{\star}_{p, \{1\}^{m}}(t)$ is defined by 
        \[
        \text{\rm \pounds}^{\star}_{p, \{1\}^{m}}(t):=\sum_{p-1 \geq n_1 \geq \cdots \geq n_m \geq 1}\frac{t^{n_1}}{n_1\cdots n_m}
        \]
        (see \ref{sec:app}).  
		Since $p/(\zeta_p-1)^{m-1}$ is congruent to $-(\zeta_p-1)^{p-m}$ mod $(\zeta_p-1)^{p-m+1}$ by Lemma \ref{lem1.7} and 
		\[
		{\text{\rm \pounds}}_{p,\{1\}^m}^{\star}\left(\frac{1}{1-z^{-1/p}}\right)
		\equiv
		\frac{1}{z-1}\text{\rm \pounds}_{p,m}(z^{1/p}) 
		\pmod{p\mathcal O_{K_{v, z}}}
		\]
		holds as the case $\Bbbk = (\{1\}^m)$ of the congruence (\ref{new fn eq}) by the assumption $p > m+1$, we have the conclusion of Proposition \ref{propD}.
	\end{proof}
	\section{Applications of Theorem \ref{thmC}}
	\label{subsec:Appriations of Theorem C}
	In this section, we give applications of Theorem \ref{thmC}.
	We prepare some lemmas about the divisibility of elements in local fields.
	For any field $L$, we denote by $\mu(L)^{(p)}$ the set of roots of unity of $L$ whose orders are prime to $p$.
	\begin{lemma}Let $L$ be a finite extension of ${\mathbb Q}_p$ and $\pi$ a prime element of $\mathcal O_L$. Let $x=1+\sum_{n=n_0}^\infty a_n\pi^n$ be a principal unit of $L$ with $a_n\in\mu(L)^{(p)}\cup\{0\}$ and let $n_0$ be the minimal positive integer such that $a_{n_0}\neq 0$. If $n_0$ is not divided by $p$ and less than $v_L(p)$, then any $p$-th root of $x$ is not contained in $L$. Here, $v_L$ is the discrete valuation of $L$ normalized as $v_L(\pi)=1$.
		\label{lem2.1}
	\end{lemma}
	\begin{proof}Assume that a $p$-th root of $x$, say $y$, is contained in $L$. Then, $y$ is also a principal unit of $L$. Write $y=1+b\pi^n$ with $n\in {\mathbb Z}_{>0}$ and $b\in \mathcal O_L^\times$. Then, $y^p=x$ is equal to
		$1+b^p\pi^{np} +pz$ for some $z\in \mathcal O_L$. Therefore, $n_0$ is divided by $p$ or greater than or equal to $v_L(p)$.
	\end{proof}
	Now, we return to the setting in Section \ref{sec:Proof of Theorem C}.
	Let $z=\zeta u\in \mathcal O_{K_v}^\times\cap (1+\mathcal O_{K_v}^\times)$
	with $u\in 1+p\mathcal O_{K_v}$ and $\zeta\in \mu(K_v)^{(p)}\setminus\{1\}$.
	Furthermore, we suppose that $u\in 1+p^2\mathcal O_{K_v} $.
	Then, we have $K_{v, z}=K_v(\mu_p)$ and $\pi:=\zeta_p-1$ is a prime element of $\mathcal O_{K_{v, z}}$.
	For each $x\in K_{v, z}$,
	we denote the $\pi$-adic expansion of $x$ by
	\[
	x=\sum_n a_n(x)\pi^n,\quad a_n(x)\in \mu(K_v)\cup \{0\}.
	\]
	Note that $a_n(x)$ is uniquely determined by $x$ and $\pi$.
	\begin{proposition}Let $x$ be a unit of $\mathcal O_{K_{v, z}}$.
		Assume that $v_p(\log_p (x))>1/(p-1)$ and $a_n(\log_p(x))\neq 0$ with some
		integer $n$ such that $1< n<p-1$. Then, any $p$-th root of $x$ is not contained in $K_{v, z}$.
		In particular, the cohomology class $\kappa_{p, K_{v, z}}(x)\otimes\zeta_p^{\otimes (m-1)}\in H^1(K_{v, z},\mathbb F_p(1)) \otimes_{\mathbb{F}_p} \mathbb{F}_p(m-1)$ is non-zero.
		\label{prop2.5}
	\end{proposition}
	\begin{proof}Let $\exp(X)$ be the formal power series $\sum_{n=0}^\infty\frac{X^n}{n!}$. Note that, if the $p$-adic valuation of $y\in {\mathbb C}_p$ is greater than $1/(p-1)$, then the series $\exp(y)$ converges. Set $x':=\exp(\log_p(x))$.
		Then, $x/x'$ is a root of unity in $K_{v, z}$ because it is contained in the kernel of $\log_p$ (cf. \cite[Proposition 5.6]{Wa}).
		Therefore, it is sufficient to show that any $p$-th root of $\zeta_p^rx'$ is not contained in $K_{v, z}$
		for any integer $r$.
		
		By the assumption for $\log_p(x)$, $a_n(x')$ is non-zero
		for some integer $n$ satisfying $1<n<p-1$.
		On the other hand, since $a_1(x')=0$, we have $a_{1}(\zeta_p^rx')\neq 0$
		for any integer $r$ prime to $p$. Hence, by applying Lemma \ref{lem2.1} to the case where $L=K_{v, z}$, $\pi=\zeta_p-1$, and $x=\zeta_p^rx'$, we see that any $p$-th root of $\zeta_p^rx'$ is not contained in $K_{v, z}$ for any integer $r$. Thus,
		we have the conclusion.
	\end{proof}
	\begin{corollary}
		Let $c$ be an element of $H^1_f(K_{v, z},\mathbb F_p(1))\otimes_{\mathbb{F}_p}\mathbb{F}_p(m-1)$ and write $\lambda_{v, z}^{(m)}(c) \in \Lambda_{K_{v, z}}(m-1)$ as
		\[
		\lambda_{v, z}^{(m)}(c)=\left(\sum_{n=1}^\infty a_n(\zeta_p-1)^n\bmod{\bigl(p\log_p(\mathcal{O}_{K_{v, z}}^{\times})\bigr)}\right)\otimes \zeta_p^{\otimes(m-1)},\quad a_n\in \mu(K_v)\cup \{0\}.
		\]
		If $a_1=0$ and $a_n\neq 0$ for some $1<n<p-1$, then $c$ is non-zero.
		\label{cor2.6}
	\end{corollary}
	\begin{proof}
		This is a direct consequence of Proposition \ref{prop2.5}.
	\end{proof}
	For the proof of Corollary \ref{corA'}, we need two more lemmas about Galois cohomology.
	\begin{lemma}Let $L$ be ${\mathbb Z}[1/p]$ or ${\mathbb Q}_p$ and let $m$ be a positive integer which is not divided by $p-1$. Then, the natural map $H_{\mathrm{cont}}^1(L, \mathbb{Z}_p(m))\rightarrow H_{\mathrm{cont}}^1(L,{\mathbb Q}_p(m))$ is injective. In particular, the $\mathbb Z_p$-module $H_{\mathrm{cont}}^1(L, \mathbb{Z}_p(m))$ is torsion free.
		\label{rem2}
	\end{lemma}
	\begin{proof}We note that the $\pi_1^{\mathrm{\acute{e}t}}(\mathrm{Spec}(L))$-invariant part of $\mu_p^{\otimes m}$ vanishes. Indeed, $\pi_1^{\mathrm{\acute{e}t}}(\mathrm{Spec}(L))$ acts on the $1$-dimensional $\mathbb F_p$-vector space $\mu_p^{\otimes m}$ via the $m$-th power $\omega^{m}$ of the Teichm\"uller character $\omega$ and the order of $\omega$ on $\pi_1^{\mathrm{\acute{e}t}}(\mathrm{Spec}(L))$ is $p-1$. By taking $\pi_1^{\mathrm{\acute{e}t}}(\mathrm{Spec}(L))$-invariant part of the exact sequence $0\rightarrow \mu_p^{\otimes m}\rightarrow {{\mathbb Q}_p}/{\mathbb{Z}_p}(m)\xrightarrow{p} {{\mathbb Q}_p}/{\mathbb{Z}_p}(m)\rightarrow 0$, we see that $H^0(L, {{\mathbb Q}_p}/{\mathbb{Z}_p}(m))$ has no non-trivial $p$-torsion element. Since any element of $H^0(L, {{\mathbb Q}_p}/{\mathbb{Z}_p}(m))$ is a $p$-power torsion element, we have $H^0(L, {{\mathbb Q}_p}/{\mathbb{Z}_p}(m))=0$. This implies that the natural homomorphism $H_{\mathrm{cont}}^1(L,\mathbb{Z}_p(m))\rightarrow H_{\mathrm{cont}}^1(L,\mathbb{Q}_p(m))$ is injective.
	\end{proof}
	\begin{lemma}Let $m$ be a positive odd integer and $p$ a prime number such that $m\not\equiv 0,1\pmod{p-1}$.
		Then, the restriction map $\mathrm{res}_m\colon H_{\mathrm{cont}}^1({\mathbb Z}[1/p],{\mathbb Z}_p(m))\to H_{\mathrm{cont}}^1({\mathbb Q}_p,{\mathbb Z}_p(m))$ is isomorphism if the composite
		\begin{equation}
			\label{eq2}
			H_{\mathrm{cont}}^1({\mathbb Z}[1/p],{\mathbb Z}_p(m))\rightarrow H_{\mathrm{cont}}^1({\mathbb Q}_p,{\mathbb Z}_p(m))\rightarrow H^1({\mathbb Q}_p,{\mathbb F}_p(m))
			\rightarrow H^1({\mathbb Q}_p(\zeta_p),{\mathbb F}_p(m))
		\end{equation}	
		is non-trivial.
		\label{rem3}
	\end{lemma}
	\begin{proof}
		By the local and global Euler--Poincar\'e characteristics of Galois cohomologies,
		the ranks of $H^1_{\mathrm{cont}}(\mathbb{Z}[1/p],{\mathbb Q}_p(m))$ and $H^1_{\mathrm{cont}}({\mathbb Q}_p,{\mathbb Q}_p(m))$
		are equal to $1$. Hence, by Lemma \ref{rem2}, $H_{\mathrm{cont}}^1({\mathbb Z}[1/p],{\mathbb Z}_p(m))$ and $H_{\mathrm{cont}}^1({\mathbb Q}_p,{\mathbb Z}_p(m))$
		are free $\mathbb Z_p$-modules of rank one.
		Therefore $\mathrm{res}_m$ is an isomorphism if and only if\[
		\overline{\mathrm{res}}_m\colon H_{\mathrm{cont}}^1({\mathbb Z}[1/p],{\mathbb Z}_p(m))\rightarrow H_{\mathrm{cont}}^1({\mathbb Q}_p,{\mathbb Z}_p(m))/pH_{\mathrm{cont}}^1({\mathbb Q}_p,{\mathbb Z}_p(m))
		\]
		is non-trivial.
		Therefore, if (\ref{eq2}) is non-trivial, then $\overline{\mathrm{res}}_m$ is also non-trivial
		because the homomorphism (\ref{eq2}) factors through $\overline{\mathrm{res}}_m$.
		This completes the proof of the lemma.
	\end{proof}
	Now, let $K={\mathbb Q}$, so $F={\mathbb Q}_p$.
	\begin{proof}[Proof of Corollary $\ref{corA'}$]
		Let us choose a specific path $\gamma$ from $\overrightarrow{01}$ to $-1$ so that $\mathrm{li}_{p,m}(-1, \gamma)$ is a continuous $1$-cocycle of $\pi_1^{\mathrm{\acute{e}t}}(\mathrm{Spec}(\mathbb Z[1/p]))$ valued in
		$\mathbb Z_p(m)$ (cf.\ Remark \ref{rem1}, \cite[Section 2, Corollary]{NW}). Since $\text{\rm \pounds}_{p,m}^{\mathrm{\acute{e}t}}(-1, \gamma)$
		is defined as the mod $p$ of $\mathrm{li}_{p,m}(-1, \gamma)$,
		the restriction of $\text{\rm \pounds}_{p,m}^{\mathrm{\acute{e}t}}(-1, \gamma)$ to $G_{{\mathbb Q}_p(\zeta_p)}$
		is contained in the image under the composite
		\[
		H^1_{\mathrm{cont}}(\mathbb{Z}[1/p],\mathbb{Z}_p(m))\rightarrow H^1_{\mathrm{cont}}({\mathbb Q}_p,\mathbb{Z}_p(m))\rightarrow H^1({\mathbb Q}_p(\zeta_p),\mathbb{F}_p(m)).
		\]
		Thus, it is sufficient to show the non-triviality of $\text{\rm \pounds}_{p,m}^{\mathrm{\acute{e}t}}(-1, \gamma)$
		in $H^1({\mathbb Q}_p(\zeta_p),\mathbb{F}_p(m))$ by Lemma \ref{rem3}.
		Note that $\text{\rm \pounds}_{p,m}^{\mathrm{\acute{e}t}}(-1, \gamma)$ is contained in $H^1_f(\mathbb Q_p(\zeta_p),\mathbb F_p(1))\otimes_{\mathbb F_p}\mathbb F_p(m-1)$.
		Then, we write $\lambda_{p, -1}^{(m)}(\text{\rm \pounds}_{p,m}^{\mathrm{\acute{e}t}}(-1, \gamma))$ as
		\begin{equation}
			\label{eq6}
			\lambda_{p, -1}^{(m)}(\text{\rm \pounds}_{p,m}^{\mathrm{\acute{e}t}}(-1, \gamma))=\left(\sum_{n=1}^\infty a_n(\zeta_p-1)^n\bmod{\bigl( p\log_p(\mathbb{Z}_p[\zeta_p]^{\times})\bigr)}\right)\otimes \zeta_p^{\otimes(m-1)}, 
		\end{equation}
		where $a_n\in \mu(\mathbb Q_p)\cup \{0\}$. According to Corollary \ref{cor2.6},
		it is sufficient to show that $a_1=0$ and $a_n\neq 0$ for some $1<n<p$.
		By Theorem \ref{thmC},
		$a_{p-m}$ in (\ref{eq6}) is congruent to
		$
		(-1)^m2^{-1}\text{\rm \pounds}_{p,m}\left(-1\right)
		$
		mod $p$.
		Here, we note that $\text{\rm \pounds}_{p,m}\left(-\zeta_p^n\right)$ is congruent to $\text{\rm \pounds}_{p,m}\left(-1\right)$
		mod $\zeta_p-1$ for any integer $n$.
		Since $1<p-m<p$, it is sufficient to show that $\text{\rm \pounds}_{p,m}\left(-1\right)\mod p$ is non-zero.
		By \cite[Lemma 4.1 (58)]{SS},
		the value $\text{\rm \pounds}_{p,m}(-1)$ is
		congruent to
		$
		\frac{1-2^{m-1}}{2^{m-2}m}B_{p-m}$
		mod $p$.
		Therefore if $p$ is regular,
		then $\text{\rm \pounds}_{p,m}(-1)\not\equiv 0\ (\mathrm{mod}\ p)$.
		This completes the proof of the corollary.
	\end{proof}
	By using Proposition \ref{propNW} and the equality (\ref{log-FMP}), we can deduce
	vanishing results for the finite star-multiple polylogarithms
	from facts about Galois cohomology.
	For example, we give another proof of a part of \cite[Proposition 4.2 (64)]{SS} which is a generalization of {\cite[Theorem 1.1 (1.1)]{ZWS}} {\em without using the functional equation}:
	\begin{corollary}
		If $m$ is an even positive integer, then we have
		\[
		\text{\rm \pounds}_{p,\{1\}^m}^{\star}(1/2):=\sum_{p-1 \geq n_1 \geq \cdots \geq n_m \geq 1}\frac{1}{2^{n_1}n_1\cdots n_m} \equiv 0 \pmod{p}
		\]
		for any prime number $p$ greater than $m+1$.
		\label{corB}
	\end{corollary}
	\begin{proof}
		Let $p$ be a prime number greater than $m+1$. As $m$ is even,
		the global Galois cohomology $H_{\mathrm{cont}}^1({\mathbb Z}[1/p],{\mathbb Z}_p(m))$ vanishes (cf.\ \cite[Theorem 1]{So}).
		Hence, $\text{\rm \pounds}_{p,m}^{\mathrm{\acute{e}t}}(-1, \gamma)$ vanishes
		because it is contained in the image under the mod $p$
		restriction map (Proof of Corollary \ref{corA'}).
		Thus, by the equality (\ref{log-FMP}) and Corollary \ref{cor2.6}, $\text{\rm \pounds}_{p,\{1\}^m}^{\star}(1/2)$ vanishes. 
	\end{proof}
	We can also obtain an another proof of a part of \cite[Theorem 3.22 (51)]{SS} by the distribution properties about the cyclotomic Soul\'e elements (\cite[Corollary 14.3.4]{W2}).
	\section{Ad\'ele-like reformulation}
	\label{sec:reformulation}
	In this section, we define $\mathcal{A}$-finite polylogarithms and $\mathcal{A}$-\'etale polylogarithms as elements of ad\'ele-like objects.
	
	For each commutative ring $R$, we define a commutative algebra $\mathcal A_R$ by
	\[
	\mathcal A_R:=\left. \left(\prod_{p}R/pR \right) \right/ \left(\bigoplus_{p}R/pR\right),
	\]
	where $p$ runs over all rational prime numbers. Then, we define {\em the $\mathcal{A}$-finite polylogarithm} $\text{\rm \pounds}_{\mathcal{A}, m}(t)$ by 
	\[
	\text{\rm \pounds}_{\mathcal{A}, m}(t) := (\text{\rm \pounds}_{p, m}(t)\bmod{p})_p \in \mathcal{A}_{\mathbb{Z}[t]}
	\]
	for a positive integer $m$.
	
	If $m \equiv m' \pmod{p-1}$, then we have $\text{\rm \pounds}_{p, m}(t) \equiv \text{\rm \pounds}_{p, m'}(t) \pmod{p}$.
	In other words, the ``weight'' of $\text{\rm \pounds}_{p, m}(t)$
	is determined up to mod $p-1$.
	However, we can define a weight of an $\mathcal{A}$-finite polylogarithm as {\em an integer}.　 This is one of advantages to considering such an ad\'ele-like framework.
	
	If $R=\mathbb Z$, we omit $\mathbb Z$ from our notation, that is,
	\[
	\mathcal A:=\mathcal{A}_{\mathbb{Z}}=\left. \left(\prod_{p}\mathbb F_p\right)\right/\left(\bigoplus_{p}\mathbb F_p\right)=\left(\prod_{p}\mathbb F_p\right)\otimes_{\mathbb Z}\mathbb Q.
	\]
	Thus, $\mathcal A$ can be regarded as a quotient of
	the ring $\mathbb A_{\mathbb Q,f}$ of finite ad\'eles of $\mathbb Q$ and we equip $\mathcal{A}$ with the quotient topology. One can check that the topology coincides with the indiscrete topology. Let us introduce Tate twists of the ring $\mathcal{A}$.
	\begin{definition}
		Let $m$ be an integer. Then, we define the topological Galois module
		$\mathcal A(m)$ by
		\[
		\mathcal A(m):=\mathbb A_{{\mathbb Q},f}(m)\otimes_{\mathbb A_{{\mathbb Q},f}}\mathcal A.
		\]
	\end{definition}
	$\mathcal{A}$-\'etale polylogarithms will be defined as
	continuous functions on the absolute Galois group of a number field
	valued in $\mathcal A(m)$.
	We equip $\prod_p\mathbb F_p(m)$ with the product topology.
	Then, the topology on $\mathcal A(m)$ coincides with the quotient topology induced by
	the natural surjection
	$
	\prod_p\mathbb F_p(m)\twoheadrightarrow \mathcal A(m).
	$
	\subsection{Cohomology of $\mathcal A(m)$}
	\label{subsec:Cohomology of A(m)}
	Let $K$ be a field of characteristic $0$.
	In this subsection, we study the continuous Galois cohomology of the $G_K$-module $\mathcal A(m)$.
	\begin{proposition}
		For each integer $m$ and each non-negative integer $i$,
		we have a natural isomorphism
		\[
		\left(\prod_pH^i(K,\mathbb F_p(m))\right)\otimes_{{\mathbb Z}}{\mathbb Q}\cong \frac{\prod_pH^i(K, \mathbb F_p(m))}{\bigoplus_p H^i\left(K,\mathbb F_p(m)\right)}.
		\]
		\label{prop4.4}
	\end{proposition}
	\begin{proof}
		It is sufficient to show that the maximal torsion subgroup
		of $\prod_pH^i(K, \mathbb F_p(m))$ coincides with $\bigoplus_p H^i\left(K,\mathbb F_p(m)\right)$.
		We denote by $A_{\mathrm{tor}}$ the maximal torsion subgroup of an abelian group $A$.
		The inclusion relation
		\[
		\bigoplus_p H^i\left(K,\mathbb F_p(m)\right)\subset\left(\prod_pH^i(K, \mathbb F_p(m))\right)_{\mathrm{tor}}
		\]
		is clear. Let us take an element $(x_p)_p\in\left(\prod_pH^i(K, \mathbb F_p(m))\right)_{\mathrm{tor}}$
		and a positive integer $N$ annihilating $(x_p)_p$.
		If $p$ is greater than $N$, the positive integer $N$ is non-zero in $\mathbb F_p$.
		Hence, the annihilation $N(x_p)_p=(Nx_p)_p=0$ implies that $x_p=0$ for all $p>N$.
		Thus, we have the converse inclusion relation.
	\end{proof}
	For a subset $\Sigma$ of the set of all prime numbers, we define the closed subgroup $C_{\Sigma}$ of $\prod_p \mathbb{F}_p$ by $C_{\Sigma} := \prod_{p \in \Sigma}\mathbb{F}_p \times \prod_{p \not \in \Sigma} \{0\}$. We call $\Sigma$ {\em a cofinite subset} of the set of all prime numbers when the complement $\Sigma^{\mathrm{c}}$ of $\Sigma$ with respect to the set of all prime numbers is finite.
	\begin{lemma}
		Every open subgroup of $\prod_p\mathbb{F}_p$ is given by the form $C_{\Sigma}$ for some cofinite subset  $\Sigma$ of the set of all prime numbers.	
		\label{open}\end{lemma}
	\begin{proof}
		Let $U$ be an open subgroup of $\prod_p \mathbb{F}_p$. Then, there exists a cofinite subset $\Sigma$ of the set of all prime numbers such that $U$  contains $C_{\Sigma}$, since such open subgroups form a fundamental system of $0$ in $\prod_p\mathbb{F}_p$. We fix such a $\Sigma$ and regard $\overline{U}:=U/C_{\Sigma}$ as a subgroup of $\bigoplus_{p \in \Sigma^{\mathrm{c}}}\mathbb{F}_p$ via the natural isomorphism $\bigoplus _{p \in \Sigma^{\mathrm{c}}}\mathbb{F}_p \xrightarrow{\sim} \prod_p\mathbb{F}_p/C_{\Sigma}$. For $p \in \Sigma^{\mathrm{c}}$, we define {\em an integer} $a_p$ by $a_p:=\prod_{q \in \Sigma^{\mathrm{c}}, q \neq p}q$ and 
		\[
		\alpha_p \colon \bigoplus_{p \in \Sigma^{\mathrm{c}}}\mathbb{F}_p \twoheadrightarrow \mathbb{F}_p\oplus \bigoplus_{\Sigma^{\mathrm{c}}\setminus \{ p \}}\{0\}\ \subset  \bigoplus_{p \in \Sigma^{\mathrm{c}}}\mathbb{F}_p 
		\]
		by $\alpha_p(x) := a_px$. For $u = (u_p)_p \in \overline{U}$, we see that $u':=(a_p^{-1}u_p)_p \in \overline{U}$. Indeed, by the Chinese remainder theorem, there exsists an integer $n$ such that $nu = u'$. Therefore, we have $\overline{U} = \bigoplus_{p \in \Sigma^{\mathrm{c}}}\alpha_p(\overline{U})$, since $u_p = \alpha_p(u')$ holds. Let $\Sigma'$ be the subset of $\Sigma^{\mathrm{c}}$ consisting of $p$ such that the $p$-component of some element of $\overline{U}$ is non-zero. Since $\alpha_p(\overline{U})$ is a subgroup of the cyclic group $\mathbb{F}_p$, $\alpha_p(\overline{U}) = \mathbb{F}_p$ if $p \in \Sigma'$ and $\alpha_p(\overline{U}) = \{0 \}$ if $p \in \Sigma^{\mathrm{c}} \setminus \Sigma'$. Hence, we have $U = C_{\Sigma \cup \Sigma'}$ and $\Sigma \cup \Sigma'$ is cofinite.
	\end{proof}
	\begin{lemma}
		Let $G$ be a profinite group and $H$ its closed subgroup. Then, $H$ is equal to the intersection of all open subgroups of $G$ containing $H$.
		\label{profinite closed}\end{lemma}
	\begin{proof}
		Let $x$ be an element of $G \setminus H$. Then, there exists an open normal subgroup $U$ of G not containing $x$ such that $xU$ and $H$ are disjoint. Since $HU$ is an open subgroup of $G$ satisfying $x \not \in HU$ and $H \subset HU$, $x$ is not an element of the intersection of all open subgroups of $G$ containing $H$.    	
	\end{proof}
	\begin{proposition}
		Every closed subgroup of $\prod_p\mathbb{F}_p$ is given by the form $C_{\Sigma}$ for some subset  $\Sigma$ of the set of all prime numbers.
		\label{closed}\end{proposition}
	\begin{proof}
		Let $C$ be a closed subgroup of $\prod_p \mathbb{F}_p$. By Lemma \ref{open} and Lemma \ref{profinite closed}, there exists a family $\{\Sigma_{\lambda}\}_{\lambda \in \Lambda}$ of cofinite subsets of the set of all prime numbers such that 
		\[
		C = \bigcap_{\lambda \in \Lambda}C_{\Sigma_{\lambda}}.
		\]
		Since $\bigcap_{\lambda \in \Lambda}C_{\Sigma_{\lambda}} = C_{\Sigma}$ where $\Sigma =\bigcap_{\lambda \in \Lambda}\Sigma_{\lambda}$, we have the conclusion.
	\end{proof}
	\begin{corollary}
		Let $C$ be a closed subgroup of $\prod_p\mathbb{F}_p$. Then, every element of $C$ has a finite order if and only if $C$ is a finite group.
		\label{finite order}\end{corollary}
	\begin{proof}
		By Proposition \ref{closed}, $C = C_{\Sigma}$ holds for some subset $\Sigma$ of all prime numbers. If $\Sigma$ is an infinite set, the element $(\varepsilon_p)_p$ of $C_{\Sigma}$ has an infinite order, where $\varepsilon_p= 1$ (resp. $0$) for $p \in \Sigma$ (resp. $p \not \in \Sigma$). Therefore, if every element of $C_{\Sigma}$ has a finite order, then $C_{\Sigma}$ is a finite group. The converse assertion is clear.
	\end{proof}
	\begin{proposition}
		Let $m$ be an integer. We assume that $K$ contains $\mu_p$ for any prime number $p$, where $\mu_p$ is the set of all $p$-th roots of unity in $\overline{K}$. Then, the natural homomorphism
		\begin{equation*}
			\iota_{\mathcal{A}}\colon \frac{\prod_pH^1(K, \mathbb F_p(m))}{\bigoplus_p H^1\left(K,\mathbb F_p(m)\right)} \longrightarrow H_{\mathrm{cont}}^1(K, \mathcal{A}(m))
		\end{equation*} 
		is injective.
		\label{inj prop}
	\end{proposition}
	\begin{proof}
		Note that 
		\[
		H_{\mathrm{cont}}^1(K, \prod_p \mathbb{F}_p (m)) \cong \mathrm{Hom}_{\mathrm{cont}}(G_K, \prod_p \mathbb{F}_p(m))
		\]
		and
		\[
		H_{\mathrm{cont}}^1(K, \mathcal{A}(m)) \cong \mathrm{Hom}_{\mathrm{cont}}(G_K, \mathcal{A}(m))
		\]
		hold by the assumption of $K$. Let $f\colon G_K \rightarrow \prod_p\mathbb{F}_p(m)$ be a continuous homomorphism such that the composite of $f$ and  $\prod_p\mathbb{F}_p(m) \rightarrow \mathcal{A}(m)$ is zero. Then, every element of $\mathrm{Im}(f)$ has a finite order. Since $f$ is a continuous homomorphism of profinite groups, $\mathrm{Im}(f)$ is a closed subgroup of $\prod_p\mathbb{F}_p(m)$. Therefore, $\mathrm{Im}(f)$ is a finite group by Corollary \ref{finite order} and this implies that $f$ is an element of  $\bigoplus_p\mathrm{Hom}_{\mathrm{cont}}(G_K, \mathbb{F}_p(m))$.
	\end{proof}
	Note that the inclusion $\iota_{\mathcal{A}}$ is not surjective in general.
	
	We introduce a correction of local conditions
	of $H_{\mathrm{cont}}^1(K,\mathcal A(m))$.
	For the rest of this subsection, we suppose that $K$ is an algebraic extension of
	${\mathbb Q}$.
	Then, for each prime number $p$,
	we put $K_p:=K\otimes_{{\mathbb Q}}{\mathbb Q}_p$ and define its subring $\mathcal O_{K_p}$
	by $\mathcal O_{K_p}:=\mathcal O_K\otimes_{{\mathbb Z}}{\mathbb Z}_p$.
	
	\begin{definition}Assume that $K$ contains $\mu_p$ for every prime number $p$.
		\begin{itemize}
			\item[(1)]For each prime number $p$, we put
			\[
			H^1_f(K_p,\mathbb F_p(m)):=\mathrm{Im}\Big(\mathcal O_{K_p}^\times/(\mathcal O_{K_p}^\times)^p\otimes_{\mathbb F_p} \mathbb F_p(m-1)\hookrightarrow H^1(K_p,\mathbb F_p(m))\Big).
			\]
			\item[(2)]We define the finite part $H^1_f(K,\mathcal A(m))$ of $H^1_{\mathrm{cont}}(K,\mathcal A(m))$ by
			\[
				H^1_f(K,\mathcal A(m)):=
				\mathrm{Ker}\left(\mathrm{Im}(\iota_{\mathcal{A}}) \rightarrow \left(\prod_{p}\frac{H^1(K_p,\mathbb F_p(m))}{H^1_f(K_p,\mathbb F_p(m))}\right)\otimes_{{\mathbb Z}}{\mathbb Q}\right).
			\]
			Here, the above homomorphism is given by Proposition \ref{prop4.4} and Proposition \ref{inj prop}.
		\end{itemize}
	\end{definition}
	By definition, we have a natural homomorphism
	\begin{equation*}
		\kappa^{-1}_{\mathcal A(m),K}\colon H^1_f(K,\mathcal A(m))\rightarrow \left(\prod_{p}\mathcal O_{K_p}^\times/(\mathcal O_{K_p}^\times)^p\otimes_{\mathbb{F}_p} \mathbb F_p(m-1)\right)\otimes_{{\mathbb Z}}{\mathbb Q}.
	\end{equation*}
	Let $M$ be a subfield of $K$ which is finite over ${\mathbb Q}$.
	Then, the $p$-adic logarithm induces a natural map
	\begin{equation*}
		\mathcal O_{M_p}^\times:=\bigoplus_{v\mid p}\mathcal O_{M_v}^\times\twoheadrightarrow \bigoplus_{v\mid p}\log_p(\mathcal{O}_{M_v}^{\times})\subset M\otimes_{{\mathbb Q}}{\mathbb Q}_p.
	\end{equation*}
	By taking the direct limit with respect to $M$,
	we obtain a natural homomorphism
	\begin{equation*}
		\log_{p,K}\colon \mathcal O_{K_p}^\times =\varinjlim_{M \subset K} \mathcal O_{M_p}^\times\rightarrow \varinjlim_{M \subset K} M\otimes_{{\mathbb Q}}{\mathbb Q}_p=K\otimes_{{\mathbb Q}}{\mathbb Q}_p.
		\label{logp}
	\end{equation*}
	\begin{remark}
		The ${\mathbb Z}_p$-module $\log_{p, K}(\mathcal{O}_{K_p}^{\times})$ is not an $\mathcal O_{K_p}$-submodule of $K_p$, in general.
	\end{remark}
	\begin{definition}
		We define the $\mathcal A$-module $\Lambda_{\mathcal{A}, K}(m-1)$ by
		\[
		\Lambda_{\mathcal A,K}(m-1):=\left(\prod_{p}\log_{p, K}(\mathcal{O}_{K_p}^{\times})/p\log_{p, K}(\mathcal{O}_{K_p}^{\times})\otimes_{\mathbb{F}_p} \mathbb F_p(m-1)\right)\otimes_{{\mathbb Z}}{\mathbb Q}.
		\]
		\label{dfn4.7}
	\end{definition}
	\begin{example}
		If $K={\mathbb Q}$, then the $\mathcal A$-module $\Lambda_{\mathcal A,{\mathbb Q}}(0)$
		is naturally isomorphic to an $\mathcal{A}$-module $\mathbf p\mathcal A_2 := \left. \left( \prod_pp\mathbb{Z}/p^2\mathbb{Z}\right) \right/ \left( \bigoplus_pp\mathbb{Z}/p^2\mathbb{Z}\right)$, because $\log_p({\mathbb Z}_p^\times )=p{\mathbb Z}_p$.
	\end{example}
	The product of $p$-adic logarithms induces a natural homomorphism
	\[
	\log_{\mathcal A(m-1),K}\colon \left(\prod_{p}\mathcal O_{K_p}^\times/(\mathcal O_{K_p}^\times)^p\otimes_{\mathbb{F}_p} \mathbb F_p(m-1)\right)\otimes_{{\mathbb Z}}{\mathbb Q}\rightarrow \Lambda_{\mathcal A,K}(m-1).
	\]
	\begin{definition}
		Assume that $K$ contains $\mu_p$ for every prime number $p$. Then, we define a group homomorphism
		\begin{equation*}
			\lambda_{\mathcal A,K}^{(m)}\colon H^1_f(K,\mathcal A(m))\rightarrow \Lambda_{\mathcal A,K}(m-1)
		\end{equation*}
		by composing $\kappa^{-1}_{\mathcal A(m),K}$ and $\log_{\mathcal A(m-1),K}$.	
		\label{dfn4.8'}
	\end{definition}
	\subsection{Ad\'ele-like reformulation of the main theorem}
	Now we give a reformulation of our main result.
	We fix a number field $K$, $z\in K\setminus\{0,1\}$,
	and an \'etale path $\gamma$ form $\overrightarrow{01}$ to $z$ in $\mathbb P^1_{\overline K}\setminus\{0,1,\infty\}$.
	Let us define the value at $z$ of {\em the $\mathcal A$-\'etale polylogarithm} $\text{\rm \pounds}^{\mathrm{\acute{e}t}}_{\mathcal A,m} (z,\gamma)$ attached to $\gamma$ to be the composite of continuous maps
	\[
	\text{\rm \pounds}^{\mathrm{\acute{e}t}}_{\mathcal A,m} (z,\gamma)\colon G_K\xrightarrow{\li^{\mathrm{\acute{e}t}}_{\mathbb A_{\mathbb Q,f},m} (z,\gamma)}\mathbb A_{\mathbb Q,f}(m)\twoheadrightarrow \mathcal A(m),
	\]
	where $\li^{\mathrm{\acute{e}t}}_{\mathbb A_{\mathbb Q,f},m} (z,\gamma)$
	is the value at $z$ of the \'etale polylogarithm attached to $\gamma$ (Remark \ref{normalization} (2)). For each $z\in K\setminus\{0, 1\}$, we put $K_z:=K(\{\mu_p, z^{1/p}\}_p)$ where $p$ runs over all prime numbers.
	\begin{proposition}
		The restriction of the $\mathcal{A}$-\'etale polylogarithm $\text{\rm \pounds}^{\mathrm{\acute{e}t}}_{\mathcal A,m} (z,\gamma)$ to the absolute Galois group $G_{K_z}$ is a continuous group homomorphism.
		Furthermore, it is contained in the finite part $H^1_f(K_z,\mathcal A(m))$.
		\label{propB'}
	\end{proposition}
	\begin{proof}
		This is a direct consequence of Proposition \ref{propB}.
	\end{proof}
	Then, our main result can be reformulate as follows:
	\begin{theorem}Let $K$ be a number field and $m$ a positive integer greater than $1$.
		Let $z$ be an element of $K\setminus\{0, 1\}$ and $\gamma$ an \'etale path
		in ${\mathbb P}^1_{\overline K}\setminus\{0,1,\infty\}$ from $\overrightarrow{01}$ to $z$.
		We denote by $z^{1/p}$ the $p$-th root of $z$ determined by $\gamma$ for each prime number $p$.
		Then, the following congruence holds in $\Lambda_{\mathcal A,K_z}(m-1):$
		\begin{equation}
			\label{eqA}
			\lambda_{\mathcal A,K_z}^{(m)}\left( \text{\rm \pounds}^{\mathrm{\acute{e}t}}_{\mathcal A,m} (z,\gamma)\right)\equiv
			\left[ \frac{(1-\zeta_{\mathbf p})^{\mathbf {p}-m}}{z-1}\text{\rm \pounds}_{\mathcal A,m}\left(z^{1/\mathbf p}\right)\right]\otimes\zeta_{\mathbf p}^{\otimes{(m-1)}} \ \ \ \Bigl( \mathrm{mod}\ (\zeta_{\mathbf p}-1)^{{\mathbf p}-m+1}\Bigr).\nonumber
		\end{equation}
		\label{thmC'}
	\end{theorem}
	\begin{remark}
		We have used $\mathbf{p}$-notations in the above theorem. See \cite[Example 3.3]{SS}.
	\end{remark}
	\begin{proof}
		It is sufficient to show that congruence for each sufficiently large $p$-components.
		Thus, the conclusion follows from Theorem \ref{thmC} directly.
	\end{proof}
	\appendix
	\section{A functional equation of finite star-multiple polylogarithms}\label{sec:app}
	In this appendix, we prove a functional equaiton of one-variable finite star-multiple polylogarithms for the proof of Theorem \ref{thmC}. 
	
	Let $k_1, \dots, k_m$ be positive integers and $\Bbbk := (k_1, \dots, k_m)$. We call such a tuple $\Bbbk$ {\em an index}. The weight $\mathrm{wt}(\Bbbk)$ of $\Bbbk$ is defined to be $k_1+\cdots+k_m$.  For a prime number $p$, {\em the one-variable finite star-multiple polylogarithms} $\text{\pounds}^{\star}_{p, \Bbbk}(t)$ and $\widetilde{\text{\pounds}}^{\star}_{p, \Bbbk}(t)$ are defined by
	\begin{equation*}
	\label{nonstar}
	\text{\pounds}_{p, \Bbbk}^{\star}(t):=\sum_{p-1\geq n_1\geq \cdots \geq n_m\geq 1}\frac{t^{n_1}}{n_1^{k_1}\cdots n_m^{k_m}}\in \mathbb{Z}_{(p)}[t]
	\end{equation*}
	and by 
	\begin{equation*}
	\label{star}
	\widetilde{\text{\pounds}}_{p, \Bbbk}^{\star}(t):=\sum_{p-1\geq n_1\geq \cdots \geq n_m\geq 1}\frac{t^{n_m}}{n_1^{k_1}\cdots n_m^{k_m}}\in \mathbb{Z}_{(p)}[t].
	\end{equation*}
	Then, we define {\em $\mathcal{A}$-finite star-multiple polylogrithms} $\text{\pounds}^{\star}_{\mathcal{A}, \Bbbk}(t)$ and $\widetilde{\text{\pounds}}^{\star}_{\mathcal{A}, \Bbbk}(t)$ by
	\[
	\text{\pounds}^{\star}_{\mathcal{A}, \Bbbk}(t) := (\text{\pounds}_{p, \Bbbk}^{\star}(t)\bmod{p})_p, \quad \widetilde{\text{\pounds}}^{\star}_{\mathcal{A}, \Bbbk}(t) := (\widetilde{\text{\pounds}}_{p, \Bbbk}^{\star}(t)\bmod{p})_p \in \mathcal{A}_{\mathbb{Z}[t]}.
	\]
	Let $\zeta_{p}^{\star}(\Bbbk) = \text{\pounds}^{\star}_{p, \Bbbk}(1)$ and $\zeta_{\mathcal{A}}^{\star}(\Bbbk) = \text{\pounds}^{\star}_{\mathcal{A}, \Bbbk}(1)$ be {\em the finite multiple zeta-star values}.
	
	The authors proved the following polynomial identity in \cite{SS}:
	\begin{theorem}[{\cite[Corollary 2.7]{SS}}]
		Let $\Bbbk = (k_1, \dots , k_m)$ be an index, $\Bbbk^{\vee} = (k_1', \dots , k_{m'}')$, and $N$ a positive integer. Then, we have the following polynomial identity$:$
		\begin{equation}
		\sum_{N\geq n_1\geq \dots \geq n_m\geq 1}(-1)^{n_1}\binom{N}{n_1}\frac{t^{n_m}}{n_1^{k_1}\cdots n_m^{k_m}}
		=\sum_{N\geq n_1\geq \dots \geq n_{m'}\geq 1}\frac{(1-t)^{n_{m'}}-1}{n_1^{k_1'}\cdots n_{m'}^{k_{m'}'}}.
		\label{1-var. Hoffman identity}\end{equation}
		Here, $\Bbbk^{\vee}$ is the Hoffman dual of $\Bbbk$ (see \cite[Definition 2.1]{SS}).
	\label{generalized Dilcher}\end{theorem}
	The case $t=1$ and $\Bbbk = m$ gives Dilcher's identity because of $(m)^{\vee}=(\overbrace{1, \dots, 1}^m)=:(\{1\}^m)$.  Since $(-1)^{n}\binom{p-1}{n}$ is congruent to $1 \bmod{p}$ for a prime number $p$ and a positive integer $n < p$, we have 
	\begin{equation}
		\widetilde{\text{\rm \pounds}}_{p, \Bbbk}^{\star}(t) \equiv  \widetilde{\text{\rm \pounds}}_{p, \Bbbk^{\vee}}^{\star}(1-t)-\zeta_{p}^{\star}(\Bbbk^{\vee}) \pmod{p}.
	\label{main thm of SS}\end{equation}
	This is a special case of main results of \cite{SS}. We recall the reversal formulas
	\begin{equation}
	\text{\rm \pounds}_{p, {\Bbbk}}^{\star}(t) \equiv (-1)^{\wt (\Bbbk )}t^{p}\widetilde{\text{\rm \pounds}}_{p, \overline{\Bbbk}}^{\star}(t^{-1}), \quad \widetilde{\text{\rm \pounds}}_{p, {\Bbbk}}^{\star}(t) \equiv (-1)^{\wt (\Bbbk )}t^{p}\text{\rm \pounds}_{p, \overline{\Bbbk}}^{\star}(t^{-1}) \pmod{p}
    \label{reversal}\end{equation}
	(cf. \cite[Proposition 3.11]{SS}). Here, $\overline{\Bbbk}$  is the reverse index.
	
	By combining the congruences (\ref{main thm of SS}) and (\ref{reversal}), we can calculate as follows:
	\begin{equation*}
	\begin{split}
	\text{\rm \pounds}_{p, \Bbbk}^{\star}(t) &\equiv (-1)^{\wt (\Bbbk)}t^{p}\widetilde{\text{\rm \pounds}}_{p, \overline{\Bbbk}}^{\star}(t^{-1}) \\ 
	&\equiv(-1)^{\wt (\Bbbk)}t^{p}\left( \widetilde{\text{\rm \pounds}}_{p, \overline{\Bbbk}^{\vee}}^{\star}(1-t^{-1})-\zeta_{p}^{\star}(\overline{\Bbbk}^{\vee})\right)\\
	&\equiv t^{p}(1-t^{-1})^{p}\text{\rm \pounds}_{p, \Bbbk^{\vee}}^{\star}((1-t^{-1})^{-1})-t^{p}\zeta_{p}^{\star}(\Bbbk^{\vee})\\
	&\equiv (t^{p}-1)\text{\rm \pounds}_{p, \Bbbk^{\vee}}^{\star}\left( \frac{t}{t-1}\right) -t^{p}\zeta_{p}^{\star}(\Bbbk^{\vee}) \pmod{p}.
	\end{split}
	\end{equation*} 
	Therefore, we see that the following congruence holds for any prime number $p$:
	\begin{equation}
	\text{\rm \pounds}_{p, \Bbbk}^{\star}(t) \equiv  (t^{p}-1)\text{\rm \pounds}_{p, \Bbbk^{\vee}}^{\star}\left( \frac{t}{t-1}\right) -t^{p}\zeta_{p}^{\star}(\Bbbk^{\vee}) \pmod{p}.
	\label{new fn eq}\end{equation}
	We can restate this by ad\'ele-like notations as follows:
	\begin{theorem}
		Let $\Bbbk$ be an index. Then, we have
		\begin{equation*}
			\text{\rm \pounds}_{\mathcal{A}, \Bbbk}^{\star}(t) = (t^{\mathbf{p}}-1)\text{\rm \pounds}_{\mathcal{A}, \Bbbk^{\vee}}^{\star}\left( \frac{t}{t-1}\right) -t^{\mathbf{p}}\zeta_{\mathcal{A}}^{\star}(\Bbbk^{\vee}).
			\label{emp fn eq}\end{equation*}
	    Here, $t^{\mathbf{p}}$ is defined to be $(t^p)_p \in \mathcal{A}_{\mathbb{Z}[t]}$.
		\label{key fn eq}\end{theorem}
	\noindent This functional equation is not written explicitly in main results of \cite{SS}.
	\section*{Acknowledgements}
	The authors would like to thank their advisor Prof.\ Tadashi Ochiai, Sho Ogaki, and Kazuma Akagi for reading the manuscript carefully. They also would like to express their gratitude to the referee for useful suggestions to improve the manuscript.

\end{document}